\documentclass[11pt,a4paper]{amsart}%
\usepackage{amsfonts,soul}
\usepackage[utf8]{inputenc}
\usepackage{amsmath,amssymb}
\usepackage[dvipsnames]{xcolor}%
\usepackage{amsmath}%
\usepackage{amssymb}%
\usepackage{graphicx}
\usepackage{enumitem}

\providecommand{\U}[1]{\protect\rule{.1in}{.1in}}
\theoremstyle{plain}
\newtheorem{theorem}{Theorem}[section]
\newtheorem{lemma}[theorem]{Lemma}

\newtheorem{remark}[theorem]{Remark}

\newtheorem{corollary}[theorem]{Corollary}
\newtheorem{proposition}[theorem]{Proposition}

\newtheorem{question}[theorem]{Question}

\def\CC{\mathbb C}

\def\U{\mathcal U}

\def\KK{\mathbb K}
\def\bi{\mathbf i}

\newcommand{\veps}{\varepsilon}

\begin{document}
	
	\title[On coincidence results for summing multilinear operators]{On coincidence results for summing multilinear operators: interpolation, $\ell_1$-spaces and cotype}

	\author{F. Bayart, D. Pellegrino, P. Rueda}
	\thanks{F. Bayart was partially supported by the grant ANR-17-CE40-0021 of the French National Research Agency ANR (project Front). D. Pellegrino was partially supported by the Réseau Franco-Brésilien en Mathématiques. P. Rueda is supported by Ministerio de Econom\'{\i}a, Industria y Competitividad and FEDER under project MTM2016-77054-C2-1-P. }

	\begin{abstract}
Grothendieck's theorem asserts that every continuous linear operator from $\ell_1$ to $\ell_2$ is absolutely $(1,1)$-summing. This kind of result is commonly called coincidence result. In this paper we investigate coincidence results in the multilinear setting, showing how the cotype of the spaces involved affect such results. The special role played by $\ell_1$ spaces is also investigated with relation to interpolation of tensor products. In particular, an open problem on the interpolation of $m$ injective tensor products is solved.
	\end{abstract}

	\maketitle

\tableofcontents

\section{Introduction}

We start  from $m\geq 1$, $X_1,\dots,X_m$, $Y$ Banach spaces over $\mathbb{K}=\mathbb{R}$ or $\mathbb{C}$ and $T:X_1\times\cdots\times X_m\to Y$ $m$-linear. When explicitly said, we will just work with complex Banach spaces. Let also $\Lambda\subset\mathbb N^m$. For $r\in(0,+\infty)$ and $p\geq 1$, we say that $T$ is $\Lambda-(r,p)-$summing if there exists a constant $C>0$ such that for all sequences $x(j)\subset X_j^{\mathbb N}$, $1\leq j\leq m$, 
$$\left(\sum_{\bi\in\Lambda}\|T(x_\bi)\|^r\right)^{\frac 1r}\leq Cw_{p}(x(1))\cdots w_{p}(x(m))$$
where $T(x_\bi)$ stands for $T(x_{i_1}(1),\dots,x_{i_m}(m))$ and $\omega_p(x)$ stands for the weak $\ell^p$-norm of $x$ defined by 
\begin{eqnarray*}
 \omega_p(x)&=&\sup_{\|x^*\|\leq 1}\left(\sum_{i=1}^{+\infty} |x^*(x_i)|^p\right)^{\frac 1p}.
\end{eqnarray*}
The least constant $C$ for which the inequality holds is denoted by $\pi_{r,p}^{\Lambda}(T)$ and the class of all $\Lambda-(r,p)-$summing multilinear maps from $X_1\times\cdots\times X_m$ to $Y$ will be denoted by $\Pi_{r,p}^{\Lambda}(^mX_1,\dots,X_m;Y)$.

When $\Lambda=\mathbb N^m$, we recover the notion of a $(r,p)-$multiple summing map introduced in \cite{BPV04} and \cite{Mat03} and we shall denote by $\Pi_{r,p}^{mult}$ the corresponding class.
When $\Lambda=\{(n,\dots,n);\ n\in\mathbb N\}$, we get the definition of a $(r,p)$-absolutely summing map which was introduced in \cite{AM}. 
We shall denote by $\Pi_{r,p}^{abs}$ this class.

Grothendieck's famous theorem for absolutely summing linear operators asserts that every continuous linear operator from $\ell_{1}$ to $\ell_{2}$ is absolutely $(1,1)$-summing. This is a coincidence result where $\ell_1$ and $\ell_2$ play a very special role. Indeed, it was shown by Lindenstrauss and Pelczynski \cite{LP68} that  if $F$ is an infinite dimensional Banach space and $E$ is an infinite 
dimensional Banach space with unconditional Schauder basis, and if every continuous linear operator from $E$ to $F$ is absolutely $(1,1)$-summing,
then $E=\ell_{1}$ and $F$ is necessarily a Hilbert space. 

Grothendieck's theorem was extended by Kwapie\'n in \cite{Kwa68} replacing $\ell_2$ by $\ell_p$. More precisely, Kwapie\'n shows there that any operator $T:\ell_{1}\to\ell_p$ is $(r,1)$-summing, with $1/r = 1 - | 1/p - 1/2|$ and that this value of $r$ is optimal. The proof of this theorem is particularly interesting because it seems to be the first time that interpolation is used in the summability theory of operators.

Our aim in this paper is to give new coincidence results for multilinear maps, focusing on general methods to obtain such results. For instance, in \cite{PGVil03}, the following multilinear extension of Grothendieck's theorem is given: any $m$-linear form $\ell_1\times\cdots\times\ell_1\to \mathbb C$ is multiple $(1,1)$-summing. We will extend this result by showing that any $m$-linear form $\ell_1\times\cdots\times\ell_1\times Z\to\mathbb C$ with $Z$ a cotype 2 space is multiple $(1,1)$-summing and that this result is optimal:  if $q_Z:=\inf\{q;\ Z\textrm{ has cotype }q\}>2$, then there exists $T:\ell_1\times\cdots\times\ell_1\times Z\to\mathbb C$ which is not $(1,1)$-summing.

The paper is organized as follows. In Section \ref{sec:interpolation} we develop a general result for interpolation of coincidence results of multilinear maps defined on a product of $\ell_1$-spaces.
We deduce from this result an improvement of a theorem of \cite{DPS10}, which is a multilinear version of Kwapie\'n theorem. These interpolation results depend on the possibility to interpolate injective
tensor products. Results are known for the tensor products of two spaces (see \cite{Ko} and \cite{Mic}) but almost nothing is known about the interpolation of the tensor products of $m$ spaces
with $m\geq 3$. We then give a negative answer to a question asked in \cite{GMM16}, showing for instance that the interpolating space of $\otimes_{m,\veps}\ell_1$ and $\otimes_{m,\veps}\ell_2$, $m\geq 3$,
is not the expected $\otimes_{m,\veps}\ell_p$.
In Section \ref{sec:grothendieck}, we will prove that any $m$-linear form $\ell_1\times\cdots\times\ell_1\times Z\to\mathbb C$ with $Z$ a cotype 2 space is multiple $(1,1)$-summing. This will need another specific property of $\ell_1$.  Section \ref{sec:cotype} is devoted to the influence of cotype in the theory of summing multilinear maps. We shall show that conditions on the cotype of the ambient spaces give restrictions on the possible values of $r$ such that any multilinear map is $\Lambda-(r,1)-$summing. In particular, we  improve several results of various authors.

 Throughout the paper, we will use the following notations. Given a Banach space $Z$, $q_Z$ will mean $\inf\{2\leq q\leq \infty: Z \mbox{ has cotype } q\}$. Henceforth if $p\in[1,\infty]$ we denote by $p^{\ast}$ its conjugate exponent. 
We will also consider the spaces
$$
\ell_p^{w}(X)=\left\{ (x_n)_{n=1}^\infty\subset X: \omega_p((x_n)_{n=1}^\infty)=\sup_{\|x^*\|\leq 1}\left(\sum_{i=1}^{+\infty} |x^*(x_i)|^p\right)^{\frac 1p}<\infty\right\}, \
$$
and
$$
\ell_p^{w,0}(X)=\Big\{ (x_n)_{n=1}^\infty\in \ell_p^w(X): \omega_p((x_n)_{n=N}^\infty) \to 0 \mbox{ for } N\to\infty\Big\}.
$$



\section{Interpolation of tensor products and coincidence results}\label{sec:interpolation}

\subsection{Interpolation of coincidence results} 
To establish a general theorem that provides coincidence results for interpolated spaces, we start by introducing the following property, which plays a central role in our work. Let $0<\theta<1$. We will say that a family of interpolation pairs of complex Banach spaces $\big((X_0(1),X_1(1)),\ldots,(X_0(m),X_1(m));(Y_0,Y_1)\big)$  has the {\it  injective $\theta$-property} if 
$$
[X_0(1)^*\widehat\otimes_\varepsilon \cdots \widehat\otimes_\varepsilon  X_0(m)^*\widehat\otimes_\varepsilon Y_0,X_1(1)^*\widehat\otimes_\varepsilon \cdots \widehat\otimes_\varepsilon X_1(m)^*\widehat\otimes_\varepsilon Y_1]_\theta =
$$
$$
[X_0(1)^*, X_1(1)^*]_\theta \widehat\otimes_\varepsilon \cdots \widehat\otimes_\varepsilon[ X_0(m)^*, X_1(m)^*]_\theta \widehat\otimes_\varepsilon [Y_0, Y_1]_\theta.
$$
 For $\theta\in ]0,1[$, we write $Y_\theta=[Y_0,Y_1]$ and  $X_\theta(j)=[X_0(j),X_1(j)]_\theta$, $j=1,\ldots,m$. 

Note that,  if $Y_0=Z_0^*$ and $Y_1=Z_1^*$ are dual Banach spaces, then the family of interpolation pairs $\big((X_0(1),X_1(1)),\ldots,(X_0(m),X_1(m));(Y_0,Y_1)\big)$ has the injective $\theta$-property if and only if $\big((X_0(1),X_1(1)),\ldots,(X_0(m),X_1(m)),(Z_0,Z_1);(\mathbb C,\mathbb C)\big)$ has it.

Kouba \cite{Ko} studied the interpolation of $2$-fold injective tensor products of complex Banach spaces and gave sufficient conditions for the interpolation pairs $(X_0,X_1),(Y_0,Y_1)$ to fulfill  $[X_0\widehat\otimes_\varepsilon Y_0,X_1\widehat\otimes_\varepsilon Y_1]_\theta \simeq [X_0,X_1]_\theta \widehat\otimes_\varepsilon [Y_0,Y_1]$. 
By \cite[Theorem 4.2]{Ko} the family $\big((X_0,X_1),(Y_0,Y_1);(\mathbb C,\mathbb C)\big)$   satisfies the injective $\theta$-property whenever 
\begin{enumerate}
\item $X_0,X_1,Y_0,Y_1$ are type  $2$ spaces or 
\item $X_0,X_1$ are type $2$ spaces and $Y_0^*,Y_1^*$ are $2$-concave Banach lattices or
\item $X_0^*,X_1^*,Y_0^*,Y_1^*$ are $2$-concave Banach lattices.
\end{enumerate}
In particular, if $0<\theta<1$ and $1\leq p_0,p_1,q_0,q_1\leq 2$,  $\frac{1}{p_\theta}=\frac{1-\theta}{p_0}+\frac{\theta}{p_1}, \frac{1}{q_\theta}=\frac{1-\theta}{q_0}+\frac{\theta}{q_1}$,  it follows that 
$$
[\ell_{p_0}\widehat\otimes_\varepsilon \ell_{q_0},\ell_{p_1}\widehat\otimes_\varepsilon \ell_{q_1}]_\theta=\ell_{p_\theta}\widehat\otimes_\varepsilon \ell_{q_\theta};
$$
in our words, the families $\big((\ell_{p_0^*},\ell_{p_1^*}),(\ell_{q_0^*},\ell_{q_1^*});(\mathbb C,\mathbb C)\big)$ or $\big((\ell_{p_0^*},\ell_{p_1^*});(\ell_{q_0},\ell_{q_1})\big)$ have the injective $\theta$-property for all $0<\theta<1$. In Theorem \ref{thm:interpolationtensor} we will see that the family $\big( (\ell_{p_0^*},\ell_{p_1^*}),\stackrel{m}{\ldots},(\ell_{p_0^*},\ell_{p_1^*});(\mathbb C,\mathbb C)\big)$ does not have the injective $\theta$-property for any $\theta\in ]0,1[$ whenever $m\geq 3$ and $1< p_0<m^*<p_1$.

\begin{theorem}\label{genint} Let $m\geq 1$, $\Lambda\subset\mathbb N^m$, $1\leq p_0,p_1\leq 2$ and  for $i=0,1$, $r_i\geq 1$.  Let $0<\theta<1$ and let $(X_0(1),X_1(1)),\ldots,(X_0(m),X_1(m)),(Y_0,Y_1)$ be  interpolation pairs of complex Banach spaces such that $\big((X_0(1),X_1(1)),\ldots,(X_0(m),X_1(m));(Y_0,Y_1)\big)$ and  $\big((\ell_{p_0^*},\ell_{p_1^*});(X_0(j),X_1(j))\big)$, $j=1,\ldots,m$,  satisfy the injective $\theta$-property. Denote by  $r_\theta$, $p_\theta$ the real numbers
$$\frac{1}{r_\theta}=\frac{1-\theta}{r_0}+\frac{\theta}{r_1},\ \frac{1}{p_\theta}=\frac{1-\theta}{p_0}+\frac{\theta}{p_1}.$$ 
  If
$$
\Pi_{r_0,p_0}^{\Lambda}(^mX_0(1),\dots,X_0(m); Y_0)=\mathcal L(^mX_0(1),\dots,X_0(m); Y_0)
$$
and
$$
\Pi_{r_1,p_1}^{\Lambda}(^mX_1(1),\dots,X_1(m); Y_1)=\mathcal L(^mX_1(1),\dots,X_1(m); Y_1)
$$
 then,
\[ \Pi_{r_\theta,p_\theta}^{\Lambda}(^mX_\theta(1),\dots,X_\theta(m); Y_\theta)=\mathcal L(^mX_\theta(1),\dots,X_\theta(m); Y_\theta).\]

\end{theorem}

\begin{proof}
Let, for $i=1,2$, 
\begin{align*}
U_i:\mathcal L(^mX_i(1),\dots,X_i(m); Y_i)&\to\mathcal L\big(^m \ell_{p_i}^{w,0}(X_i(1)),\dots,\ell_{p_i}^{w,0}(X_i(m));\ell_{r_i}(Y_i)\big)\\
T&\mapsto \widehat T
\end{align*}
where $\widehat T$ is defined by
\[ \widehat T\Big(\big(x_i(1)\big)_i,\dots,\big(x_i(m)\big)_i\Big)=\Big(T\big(x_{i_1}(1),\dots,x_{i_m}(m)\big)\Big)_{i\in\Lambda}. \]
Our assumption tells us that $U_1$ and $U_2$ are bounded maps. Hence, it induces a bounded map $U_\theta$ from $\big
[\mathcal L(^mX_0(1),\dots,X_0(m);Y_0),\mathcal L(^mX_1(1),\dots,X_1(m);Y_1)\big]_\theta$ into 
$$\Big[ \mathcal L\big(^m \ell_{p_0}^{w,0}(X_0(1)),\dots,\ell_{p_0}^{w,0}(X_0(m));\ell_{r_0}(Y_0)\big),
 \mathcal L\big(^m \ell_{p_1}^{w,0}(X_1(1)),\dots,\ell_{p_1}^{w,0}(X_1(m));\ell_{r_1}(Y_1)\big) \Big]_\theta.$$ 

Since $\mathcal L(^mX_i(1),\ldots,X_i(m);Y_i)\simeq X_i(1)^*\widehat\otimes_\varepsilon \cdots \widehat\otimes_\varepsilon  X_i(m)^*\widehat\otimes_\varepsilon Y_i$, by the injective $\theta$-property we get that 

 $$\Big
[\mathcal L\big(^mX_0(1),\dots,X_0(m);Y_0\big),\mathcal L\big(^mX_1(1),\dots,X_1(m);Y_1\big)\Big]_\theta=\mathcal L(^mX_\theta(1),\ldots,X_\theta(m);Y_\theta).$$
Let us denote
$$
E=\Big[ \mathcal L\big(^m \ell_{p_0}^{w,0}(X_0(1)),\dots,\ell_{p_0}^{w,0}(X_0(m));\ell_{r_0}(Y_0)\big),
 \mathcal L\big(^m \ell_{p_1}^{w,0}(X_1(1)),\dots,\ell_{p_1}^{w,0}(X_1(m));\ell_{r_1}(Y_1)\big) \Big]_\theta
$$
and
$$
F=\mathcal L\Big(^m\big[\ell_{p_0}^{w,0}(X_0(1)),\ell_{p_1}^{w,0}(X_1(1))\big]_\theta,\ldots ,\big[\ell_{p_0}^{w,0}(X_0(m)),\ell_{p_1}^{w,0}(X_1(m))\big]_\theta;\big[\ell_{r_0}(Y_0),\ell_{r_1}(Y_1)\big]_\theta\Big).
$$
By interpolation of multilinear maps we know that 
$$
\|\widehat T\|_F\leq \|\widehat T\|_E.
$$

Now, $\ell_p^{w,0}(X)\simeq \ell_p\widehat \otimes_\varepsilon X$. Since $\big((\ell_{p_0^*},\ell_{p_1^*});(X_0(j),X_1(j))\big)$ satisfies the injective $\theta$-property for all $j=1,\ldots,m$, 
 it follows that $\big[\ell_{p_0}^{w,0}(X_0(j)),\ell_{p_1}^{w,0}(X_1(j))\big]_\theta=\ell_{p_\theta}^{w,0}(X_\theta(j))$. By interpolation of vector-valued $\ell_p$-spaces, we also know that $\big[\ell_{r_0}(Y_0),\ell_{r_1}(Y_1)\big]_\theta=\ell_{r_\theta}(Y_\theta)$. Hence, 
$$
U_\theta:\mathcal L(^mX_\theta(1),\ldots,X_\theta(m);Y_\theta)\to \mathcal L\big(^m\ell_{p_\theta(1)}^{w,0}(X_\theta(1)),\ldots,\ell_{p_\theta(m)}^{w,0}(X_\theta(m));\ell_{r_\theta}(Y_\theta)\big)
$$
is continuous, and the results follows.
\end{proof}

The following is a particular case of the above general theorem.

\begin{theorem}\label{thm:interpolation}
Let $(X_0,X_1)$ be an interpolation pair of Banach spaces. Let for $i=0,1$, $r_i\geq p_i\geq 1$, let $m\geq 1$ and let $\Lambda\subset\mathbb N^m$. For $\theta\in ]0,1[$, denote by $X_\theta=[X_0,X_1]_\theta$ and by $r_\theta$, $p_\theta$ the real numbers
$$\frac{1}{r_\theta}=\frac{1-\theta}{r_0}+\frac{\theta}{r_1},\ \frac{1}{p_\theta}=\frac{1-\theta}{p_0}+\frac{\theta}{p_1}.$$
\begin{enumerate}
\item[(i)] Assume that $1\leq p_0,p_1\leq 2$, $m\geq 2$ and that, either $X_0^*$ and $X_1^*$ are type 2 spaces or $X_0,X_1$ are $2$-concave Banach lattices. If 
\begin{align*}
\Pi_{r_0,p_0}^{\Lambda}(^m\ell_1,\dots,\ell_1,X_0;\CC)&=\mathcal L(^m\ell_1,\dots,\ell_1,X_0;\CC)\\
\Pi_{r_1,p_1}^{\Lambda}(^m\ell_1,\dots,\ell_1,X_1;\CC)&=\mathcal L(^m\ell_1,\dots,\ell_1,X_1;\CC),\\
\end{align*}
then
\[ \Pi_{r_\theta,p_\theta}^{\Lambda}(^m\ell_1,\dots,\ell_1,X_\theta;\CC)=\mathcal L(^m\ell_1,\dots,\ell_1,X_\theta;\CC).\]
\item[(ii)] Assume that $1\leq p_0,p_1\leq 2$ or that $p_0=p_1$. If 
\begin{align*}
\Pi_{r_0,p_0}^{\Lambda}(^m\ell_1,\dots,\ell_1;X_0)&=\mathcal L(^m \ell_1,\dots,\ell_1;X_0)\\
\Pi_{r_1,p_1}^{\Lambda}(^m\ell_1,\dots,\ell_1;X_1)&=\mathcal L(^m \ell_1,\dots,\ell_1;X_1),\\
\end{align*}
then
\[\Pi_{r_\theta,p_\theta}^{\Lambda}(^m\ell_1,\dots,\ell_1;X_\theta)=\mathcal L(^m \ell_1,\dots,\ell_1;X_\theta).\]
\end{enumerate} 
\end{theorem}

\begin{proof} We start with part (i). 
By Theorem \ref{genint} it suffices to prove  that the families 
\begin{equation}\label{ellX}
\big((\ell_1,\ell_1),\ldots,(\ell_1,\ell_1),(X_0,X_1);(\mathbb C,\mathbb C)\big),
\end{equation}
\begin{equation}\label{ellX2}
\big((\ell_{p_0^*},\ell_{p_1^*});(X_0,X_1)\big)
\end{equation}
and 
\begin{equation}\label{ellX3}
\big((\ell_{p_0^*},\ell_{p_1^*});(\ell_1,\ell_1)\big)
\end{equation}
 satisfy the injective $\theta$-property for any $0<\theta<1$. 

For the family (\ref{ellX}) we have
$$
[\ell_\infty\widehat\otimes_\varepsilon\cdots\widehat\otimes_\varepsilon\ell_\infty\widehat\otimes_\varepsilon X_0^*,\ell_\infty\widehat\otimes_\varepsilon\cdots\widehat\otimes_\varepsilon\ell_\infty\widehat\otimes_\varepsilon X_1^*]_\theta =
 [\ell_\infty\widehat\otimes_\varepsilon X_0^*,\ell_\infty\widehat\otimes_\varepsilon X_1^*]_\theta
=
$$
$$
 [\ell_\infty( X_0^*),\ell_\infty(X_1^*)]_\theta=\ell_\infty(X_\theta^*)=\ell_\infty\widehat\otimes_\varepsilon X_\theta^*=\ell_\infty\widehat\otimes_\varepsilon\cdots\widehat\otimes_\varepsilon\ell_\infty\widehat\otimes_\varepsilon X_\theta^*,
$$
where above we have used the well-known fact that, for any Banach space $X$, we have $ \ell_\infty\widehat\otimes_\veps\cdots\widehat \otimes_\veps\ell_\infty\widehat\otimes_\veps X^*\simeq\ell_\infty\widehat\otimes_\varepsilon X^*=\ell_\infty(X^*)$. 

By \cite[Theorem 4.2]{Ko} and the assumptions on $X_i(j)^*$ or on $X_i(j)$, $i=0,1$, $j=1,\ldots,m$,   we have that the family (\ref{ellX2}) satisfies the injective $\theta$-property. The family (\ref{ellX3})  satisfies the property by \cite[Theorem 4.2]{Ko}.

To get part (ii), it remains to observe that under our assumptions on $p_0$ and $p_1$, the families $\big((\ell_1,\ell_1),\ldots,(\ell_1,\ell_1);(X_0,X_1)\big)$ and $\big(\ell_{p_0^*},\ell_{p_1^*});(\ell_1,\ell_1)\big)$ have the injective $\theta$-property.
By the interpolation of vector-valued $\ell_p$-spaces, 
$$
[\ell_\infty\widehat\otimes_\varepsilon\cdots\widehat\otimes_\varepsilon\ell_\infty\widehat\otimes_\varepsilon X_0,\ell_\infty\widehat\otimes_\varepsilon\cdots\widehat\otimes_\varepsilon\ell_\infty\widehat\otimes_\varepsilon X_1]_\theta =
 [\ell_\infty( X_0),\ell_\infty(X_1)]_\theta=
$$
$$
\ell_\infty(X_\theta)=\ell_\infty\widehat\otimes_\varepsilon X_\theta=\ell_\infty\widehat\otimes_\varepsilon\cdots\widehat\otimes_\varepsilon\ell_\infty\widehat\otimes_\varepsilon X_\theta.
$$
\end{proof}

\subsection{A multilinear Kwapie\'n's theorem}
As an application of the previous result, we improve a multilinear version of Kwapie\'n's theorem which was proposed in \cite{DPS10} in the following form. Let  $T\in\mathcal{L}(^{m}\ell_{1};\ell_{p})$ and $A_{k}\in\mathcal{L}(^{n}
\ell_{\infty};\ell_{1})$ for all $k=1,...,m$. Then the composition $T\left(
A_{1},...,A_{m}\right)  $ is multiple $(r,1)$-summing for
\[
r=\left\{
\begin{array}
[c]{ll}%
\frac{2n}{n+2-\frac{2}{p}} & \text{ if }1\leq p\leq2\\
\frac{2n}{\frac{2n}p+1} & \text{ if }2\leq p\leq \frac{2n}{n-1}\\
2&\ \text{ if }\frac{2n}{n-1}\leq p\leq\infty.
\end{array}
\right.
\]

We improve this result when $p\geq 2$ and we also give an analogue for the notion of absolute summability.

\begin{theorem}\label{thm:kwapien}
Let $T\in\mathcal{L}(^{m}\ell_{1};\ell_{p})$ and $A_{k}\in\mathcal{L}(^{n}
\ell_{\infty};\ell_{1})$ for all $k=1,...,m$. 
\begin{enumerate}
\item The composition $T\left(
A_{1},...,A_{m}\right)  $ is multiple $(r,1)$-summing for
\[
r=\left\{
\begin{array}
[c]{ll}%
\frac{2n}{n+2-\frac{2}{p}} & \text{ if }1\leq p\leq2\\
\frac{2n}{n+\frac{2}{p}} & \text{ if }2\leq p\leq+\infty.
\end{array}
\right.
\]
\item Assume $n\geq 2$. The composition $T\left(
	A_{1},...,A_{m}\right)  $ is absolutely $(r,2)$-summing for
\[
r=\left\{
\begin{array}
[c]{ll}%
\frac{2p}{mp+2p-2} & \text{ if }1\leq p\leq2\\
\frac{2p}{mp+2} & \text{ if }2\leq p\leq+\infty.
\end{array}
\right.
\]	
\end{enumerate}
\end{theorem}
Observe that for $n=m=1$, the first point gives exactly Kwapie\'n's theorem.
\begin{proof}
By \cite[Theorem 6.1]{DPS10}, the theorem is known for $p\leq2$ and for
$p=+\infty$. The statement for $p\in (2,+\infty)$ follows from a variant of
Theorem \ref{thm:interpolation}.  We fix
$A_{1},...,A_{m}$ and define, for $p\geq2$, $r(p)=\frac{2n}{n+\frac{2}{p}}$.
Let $S_p$ be defined by
\begin{align*}
S_p:\mathcal L(^m\ell_1,\dots,\ell_1;\ell_p)&\to\mathcal L(^m \ell_{1}^w(\ell_\infty),\dots,\ell_{1}^w(\ell_\infty);\ell_{r(p)})\\
T&\mapsto \widehat T_A
\end{align*}
where $\widehat T_A$ is defined by
\[ \widehat T_A\big(x_i(1),\dots,x_i(m)\big)=\big(T(A_1(x_{i_1}(1)),\dots,A_m(x_{i_m}(m))\big)_{i\in\Lambda}. \]
We know that $S_{2}$ and $S_{\infty}$ are bounded and we may interpolate as in the proof of Theorem~\ref{thm:interpolation} to deduce that $S_p$ is bounded for all $p\in [2,+\infty]$.

In order to prove (2) we shall use the following three results:
\begin{enumerate}[label=(\alph*)]
	
\item (\cite[Corollary 2.5]{perez}) Every continuous $n$-linear ($n\geq2$) form on $\ell_{\infty}\times\ldots\times\ell_{\infty}$ is absolutely $(1;2)$-summing. 
\item (\cite[Theorem 3.1]{bern}) Every continuous $m$-linear operator from $\ell_{1}\times\ldots\times\ell_{1}$ to $\ell_{2}$ is absolutely $(2/(m+1);1)$-summing.
\item (\cite[Theorem 2.5]{bot}) Every continuous $m$-linear operator from $\ell_{1}\times\ldots\times\ell_{1}$ to any Banach space $F$ is absolutely $(2/m;1)$-summing. 
\end{enumerate}

	Let us consider the case $1\leq p\leq 2$. From (b) and (c) with $F=\ell_{1}$ and Theorem \ref{thm:interpolation} we conclude that every continuous $m$ linear operator from $\ell_{1}\times\ldots\times\ell_{1}$ to $\ell_{p}$ is absolutely $(2p/(mp+2p-2);1)$-summing.
	Now the result is completed since from (a) we know that all continuous $n$-linear ($n\geq2$) operators from $\ell_{\infty}\times\ldots\times\ell_{\infty}$ to $\ell_{1}$ are weakly absolutely $(1;2)$-summing (sends weakly 2 summable sequences into weakly 1 summable sequences).

	Let us consider the case $2\leq p\leq \infty$. From (b) and (c) with $F=\ell_{\infty}$ and Theorem \ref{thm:interpolation} we conclude that every continuous $m$ linear operator from $\ell_{1}\times\ldots\times\ell_{1}$ to $\ell_{p}$ is absolutely $(2p/(mp+2);1)$-summing. 
	Now the result is completed since from (a) all continuous $n$ linear ($n\geq2$) operators from $\ell_{\infty}\times\ldots\times\ell_{\infty}$ to $\ell_{1}$ are weakly absolutely $(1;2)$-summing (sends weakly 2 summable sequences into weakly 1 summable sequences).
\end{proof}

\begin{remark}
There is an interesting phenomenon here: while in the case of multiple summing operators we have dependence just on $n$ in the final result, we have dependence just on $m$ in the case of absolutely summing operators.
\end{remark}

\subsection{Interpolation of $m$ injective tensor products}

The proof of Theorem \ref{genint}  depends heavily on the following formula of interpolation:
$$
[X_0(1)^*\widehat\otimes_\varepsilon \cdots \widehat\otimes_\varepsilon  X_0(m)^*\widehat\otimes_\varepsilon Y_0,X_1(1)^*\widehat\otimes_\varepsilon \cdots \widehat\otimes_\varepsilon X_1(m)^*\widehat\otimes_\varepsilon Y_1]_\theta =
$$
$$
[X_0(1)^*, X_1(1)^*]_\theta \widehat\otimes_\varepsilon \cdots \widehat\otimes_\varepsilon[ X_0(m)^*, X_1(m)^*]_\theta \widehat\otimes_\varepsilon [Y_0, Y_1]_\theta
$$
that defines the injective $\theta$-property for the family of interpolation pairs of Banach spaces $\big((X_0(1),X_1(1)),\ldots,(X_0(m),X_1(m));(Y_0,Y_1)\big)$. 
Let us pay attention to the interpolation of  injective $m$-fold tensor products.
In particular, in the proof of Theorem \ref{thm:interpolation} this formula becomes 
$$[\ell_\infty\widehat\otimes_\veps\cdots\widehat\otimes_\veps\ell_\infty\widehat\otimes_\veps X_0,\ell_\infty\widehat\otimes_\veps\cdots\widehat\otimes_\veps\ell_\infty\widehat\otimes_\veps X_1]_\theta=\ell_\infty\widehat\otimes_\veps\cdots\widehat\otimes_\veps\ell_\infty\widehat\otimes_\veps [X_0,X_1]_\theta,$$ which is fulfilled thanks to deal with $\ell_\infty=\ell_1^*.$
To expect to replace $\ell_1$ by other spaces in Theorem \ref{thm:interpolation}, we need results saying how to interpolate injective tensor products with $m\geq 2$ factors. 
For $m=2$, this has been thoroughly studied in \cite{Ko} and \cite{Mic}. Nevertheless, nothing seems known about the interpolation of $m$ tensor products, $m\geq 3$
except when all but one of the spaces are equal to $\ell_\infty$. In particular, in \cite{GMM16}, the authors ask the following question: let $m\geq 3$, $1\leq p_0<p_1\leq 2$, $\theta\in (0,1)$. Is it true that
\[
\lbrack\widehat\otimes_{i=1,\varepsilon}^{m}\ell_{p_{0}},\widehat\otimes_{i=1,\varepsilon}%
^{m}\ell_{p_{1}}]_{\theta}=\widehat\otimes_{i=1,\varepsilon}^{m}\ell_{p_{\theta}}%
\]
where $\frac{1}{p_{\theta}}=\frac{1-\theta}{p_{0}}+\frac{\theta}{p_{1}}$? In other words: Has the family $\big((\ell_{p_0},\ell_{p_1}),\stackrel{m}{\ldots},(\ell_{p_0},\ell_{p_1})\big)$ got the injective $\theta$-property? A positive answer would have interesting consequences (see \cite{GMM16} again). Unfortunately, we
show that this is false provided $p_0$ is small enough and $p_1$ is big enough.

\begin{theorem}\label{thm:interpolationtensor}
Let $m\geq3$ be a positive integer and let $1\leq p_{0}<m^{\ast}<p_1$. Then for all $\theta\in(0,1)$
\[
\lbrack\widehat\otimes_{i=1,\varepsilon}^{m}\ell_{p_{0}},\widehat\otimes_{i=1,\varepsilon}%
^{m}\ell_{p_{1}}]_{\theta}
\ncong
\widehat\otimes_{i=1,\varepsilon}^{m}\ell_{p_{\theta}},
\]
where $\frac{1}{p_{\theta}}=\frac{1-\theta}{p_{0}}+\frac{\theta}{p_{1}}$.
\end{theorem}
For the proof, we will use the following lemma.
\begin{lemma}\label{lem:interpolationtensor}
Let $m\geq 2$, $l\in\{2,\dots,m\}$ and $i_1,\dots,i_{l-1}\in\mathbb N$ be such that
\[m^{i_1}+\cdots+m^{i_{l-1}}=(l-1)m^{i_l}.\]
Then $i_1=\cdots=i_l$.
\end{lemma}
\begin{proof}
The proof is done by induction on $l$. The case $l=2$ is easy (because $m\geq 2$). Assume that the result has
been shown for $l-1\in\{2,\dots,m-1\}$ and let us prove it for $l$. If one of $i_1,\dots,i_{l-1}$ is greater than or equal to $i_l+1$, then 
$m^{i_1}+\cdots+m^{i_{l-1}}\geq m^{i_l+1}>(l-1)m^{i_l}$, a contradiction. In the same vein, it is impossible that $i_1,\dots,i_{l-1}$ are all smaller than or equal to $i_l-1$, 
otherwise $m^{i_1}+\cdots+m^{i_{l-1}}<(l-1)m^{i_l}$. Thus, at least one $i_1,\dots,i_{l-1}$ is equal to $i_l$ and we can conclude by the induction hypothesis. 
\end{proof}

\begin{proof}[Proof of Theorem \ref{thm:interpolationtensor}]
Interpreting $X_1^*\widehat\otimes_\veps\cdots\widehat\otimes_\veps X_m^*$ as $\mathcal L(^m X_1,\dots,X_m;\mathbb C)$, 
it is sufficient to prove that, for all $q_0,q_1$ with $q_0>m$ and $q_1<m$, there exists a sequence $(T_n)$ in $\mathcal L(^m \ell_{q_0})\cap \mathcal L(^m \ell_{q_1})$ such that
$$\frac{\|T_n\|_E}{\|T_n\|_{\mathcal L(^m \ell_{q_\theta})}}\xrightarrow{n\to+\infty}+\infty$$
where $E:=\big[\mathcal L(^m\ell_{q_0}),\mathcal L(^m\ell_{q_1})\big]_\theta$.
Let $T_n(x(1),\ldots,x(m))=\sum_{i=1}^n x_i(1)\cdots x_i(m)$. It is easy to show that 
\begin{equation}\label{eq:interpolationtensor1}
\|T_n\|_{\mathcal L(^m\ell_q)}\leq
\left\{
\begin{array}{ll}
1&\textrm{ provided }q<m\\
n^{1-\frac mq}&\textrm{ provided }q\geq m.
\end{array}\right.
\end{equation}
Let us now compute $\|T_n\|_E$. We shall use standard notations for interpolation, as they are exposed in \cite{BeLo}. We recall that $S$ is the (open) strip $\{z;\ 0<\Re e(z)<1\}$ and if $\overrightarrow{X}=(X_0,X_1)$ is an interpolation couple, $\mathcal F(\overrightarrow{X})$ stands for the set of functions with values in $X_0\cap X_1$ which are bounded and continuous on the closed strip $\overline{S}$ and analytic inside $S$. Here we will always have $\overrightarrow{X}=(\mathcal L(^m \ell_{q_0}),\mathcal L(^m\ell_{q_1}))$.

We recall that
\[ \|T_n\|_E=\inf\left\{\max\left(\sup_{t\in\mathbb R}\|f(it)\|_{\mathcal L(^n \ell_{q_0})},\sup_{t\in\mathbb R}\|f(1+it)\|_{\mathcal L(^n \ell_{q_1})}\right);\ f\in\mathcal F(\overrightarrow{X}),\ f(\theta)=T_n\right\}. \]

Let $f\in\mathcal F(\overrightarrow{X})$ with $f(\theta)=T_n$. For each $z\in \overline S$ the multilinear form $f(z)$ may be written 
\[ f(z)(x(1),\dots,x(m))=\sum_{\bi\in\mathbb N^m}a_\bi(z)x_{i_1}(1)\cdots x_{i_m}(m) \]
where each $a_\bi:\overline S\to\mathbb C$ is continuous and analytic inside $S$. 
Moreover, $a_{i,\dots,i}(\theta)=1$ for all $i=1,\dots,n$, $a_\bi(\theta)=0$ if $\bi\neq (i,\dots,i)$ for some $i\in\{1,\dots,n\}$.
We then define, for $\alpha\in\mathbb R$ and $z\in\overline{S}$,
\begin{align*}
&g(x(1),\dots,x(m))(\alpha,z)\\
&\quad\quad=\sum_{\bi\in\{1,\dots,n\}^m}a_{\bi}(z)\big(e^{im^{i_1}\alpha}x_{i_1}(1)\big)\cdots\big(e^{im^{i_{m-1}}\alpha}x_{i_{m-1}}(m-1)\big)\big(e^{-i(m-1)m^{i_m}\alpha}x_{i_m}(m)\big)\\
&\quad\quad=\sum_{\bi\in\{1,\dots,n\}^m}a_{\bi}(z)e^{i(m^{i_1}+\dots+m^{i_{m-1}}-(m-1)m^{i_m})\alpha}
x_{i_1}(1)\cdots x_{i_m}(m).
\end{align*}
Then for all $\alpha\in\mathbb R$ and all $t\in\mathbb R$, $\|g(\alpha,it)\|_{\mathcal L(^m \ell_{q_0})}\leq
\|f(it)\|_{\mathcal L(^m \ell_{q_0})}$ and  $\|g(\alpha,1+it)\|_{\mathcal L(^m \ell_{q_1})}\leq
\|f(1+it)\|_{\mathcal L(^m \ell_{q_1})}$. We then set
\begin{align*}
h(x(1),\dots,x(m))(z)&=\int_0^{2\pi}g(x(1),\dots,x(m))(\alpha,z)\frac{d\alpha}{2\pi}.
\end{align*}
Then, again, 
 $\|h(it)\|_{\mathcal L(^m \ell_{q_0})}\leq 
\|f(it)\|_{\mathcal L(^m \ell_{q_0})}$ and $\|h(1+it)\|_{\mathcal L(^m \ell_{q_1})}\leq 
\|f(1+it)\|_{\mathcal L(^m \ell_{q_1})}$. Moreover, the orthogonality of the characters on the torus together with Lemma \ref{lem:interpolationtensor} ensures that 
$$h(x(1),\dots,x(m))(z)=\sum_{i=1}^n a_{i,\dots,i}(z) x_i(1)\cdots x_i(m).$$
In particular, this yields $h(\theta)=T_n$. Therefore, up to now, we have established that 
\begin{align*}
\|T_n\|_E=\inf\Big\{&\max\left(\sup_{t\in\mathbb R}\|h(it)\|_{\mathcal L(^m \ell_{q_0})},\sup_{t\in\mathbb R}\|h(1+it)\|_{\mathcal L(^m \ell_{q_1})}\right);\\
&\quad \quad h(z)=\sum_{i=1}^n a_{i}(z)x_i(1)\cdots x_i(m),\\
&\quad \quad a_i:\overline{S}\to\mathbb C\textrm{ continuous and analytic inside }S,\ a_i(\theta)=1\Big\}.
\end{align*}
Let us now define, for a multilinear map $L:\ell_q\times\cdots\times\ell_q\to\CC$ which may be written $L(x(1),\ldots,x(m))=\sum_{i=1}^n a_i x_1(1)\cdots x_i(m)$
the multilinear map $L_{\textrm{sym}}$ given by
\begin{align*} L_{\textrm{sym}}(x(1),\ldots,x(m))
 &=\frac 1{n!}\sum_{\sigma\in S_n}L(x_\sigma(1),\cdots,x_{\sigma}(m))\\
 &=\frac{1}{n!}\sum_{i=1}^n \sum_{\sigma\in S_n}a_i x_{\sigma(i)}(1)\cdots x_{\sigma(i)}(m)
\end{align*}
where $S_n$ stands for the set of permutations of $\{1,\dots,n\}$ and for a vector $x$ in $\ell_q$, $x_\sigma=(x_{\sigma(1)},\dots,x_{\sigma(n)},0,\dots)$.
Namely, we have symmetrized $L$ with respect to the coordinates of each $x(1),\dots,x(m)$ ($L$ was already symmetric if we just looked at the variables). It is straighforward to check
that $L_{\textrm{sym}}$ writes
$$L_{\textrm{sym}}(x(1),\ldots,x(m))=\sum_{j=1}^n c_j x_j(1)\cdots x_j(n)$$
where, for $j\in\{1,\dots,n\}$, $c_j$ is given by
$$c_j=\frac 1{n!}\sum_{i=1}^n \sum_{\substack{\sigma\in S_n;\\ \sigma(i)=j}}a_i.$$
Let now $k\in\{1,\dots,n\}$ and $\tau\in S_n$ with $\tau(j)=k$. Then, since $\sigma\in S_n\mapsto \tau^{-1}\circ \sigma$ is a bijection of $S_n$, 
\begin{align*}
 c_j&=\frac1{n!}\sum_{i=1}^n \sum_{\substack{\sigma\in S_n;\\ \tau^{-1}\circ\sigma(i)=j}}a_i\\
 &=\frac1{n!}\sum_{i=1}^n \sum_{\substack{\sigma\in S_n;\\ \sigma(i)=k}}a_i\\
 &=c_k.
\end{align*}
This means that all the $c_j$, $j=1,\dots,n$ are equal, so that $L_{\textrm{sym}}=bT_n$ for some $b\in\CC$. Moreover, $\|L_{\textrm{sym}}\|_{\mathcal L(^m\ell_q)}\leq \|L\|_{\mathcal L(^m\ell_q)}.$
Hence, if we start from $h(z)=\sum_{i=1}^n a_{i}(z)x_i(1)\cdots x_i(m)$ with $a_i:\overline{S}\to\mathbb C$ continuous and analytic inside $S$, $a_i(\theta)=1$,
we get $h_{\textrm{sym}}(z)=b(z)T_n$ where $b:\overline{S}\to\mathbb C$ is continuous and analytic inside $S$, $b(\theta)=1$, and 
$|b(it)|\cdot \|T_n\|_{\mathcal L(^m \ell_{q_0})}\leq \|h(it)\|_{\mathcal L(^m \ell_{q_0})}$, 
$|b(1+it)|\cdot \|T_n\|_{\mathcal L(^m \ell_{q_1})}\leq \|h(1+it)\|_{\mathcal L(^m \ell_{q_1})}$. In other words, we have shown that
\begin{align*}
\|T_n\|_E=\inf\Big\{&\max\left(\sup_{t\in\mathbb R}|b(it)|\|T_n\|_{\mathcal L(^m \ell_{q_0})},\sup_{t\in\mathbb R}|b(1+it)|\|T_n\|_{\mathcal L(^m \ell_{q_1})}\right);\\
&\quad \quad b:\overline{S}\to\mathbb C\textrm{ continuous and analytic inside }S,\ b(\theta)=1\Big\}.
\end{align*}
By the three-lines theorem,
\begin{equation}
\label{eq:interpolationtensor2}
\|T_n\|_E=\|T_n\|_{\mathcal L(^m \ell_{q_0})}^{1-\theta}\|T_n\|_{\mathcal L(^m \ell_{q_1})}^\theta=n^{(1-\theta)\left(1-\frac m{q_0}\right)}.
\end{equation}
Now, if we compare \eqref{eq:interpolationtensor1} and \eqref{eq:interpolationtensor2}, then we see that
$$\frac{\|T_n\|_E}{\|T_n\|_{\mathcal L(^m \ell_{q_\theta})}}\xrightarrow{n\to+\infty}+\infty$$
as required to conclude.
\end{proof}

\begin{question}
Assume that $1\leq p_0,p_1\leq m^*$. Do we have 
\[
\lbrack\otimes_{i=1,\varepsilon}^{m}\ell_{p_{0}},\otimes_{i=1,\varepsilon}%
^{m}\ell_{p_{1}}]_{\theta}=\otimes_{i=1,\varepsilon}^{m}\ell_{p_{\theta}}?
\]
\end{question}

\section{Multiple $(r,1)$-summing multilinear forms on products of $\ell_{1}$-spaces} \label{sec:grothendieck}

In \cite{BoPe10}, the authors have shown that, for $r\geq s\geq 1$ and $p=\min(s,2)$, if $\mathcal L(\ell_1;F)=\Pi_{r,s}^{abs}(\ell_1;F)$, 
then $\mathcal L(^m\ell_1,\dots,\ell_1;F)=\Pi^{mult}_{r,p}(^m \ell_1,\dots,\ell_1;F)$. Observe that this implies the extension of \cite{PGVil03} of the Grothendieck's inequality. We need a variant of this result.

\begin{theorem}\label{thm:repeatedl1}
 Let $r\geq s\geq 1$ and $p=\min(s,2)$. If $\mathcal L(^2\ell_1,E;F)=\Pi_{r,s}^{mult}(^2\ell_1,E;F)$, then 
$\mathcal L(^m\ell_1,\dots,\ell_1,E;F)=\Pi^{mult}_{r,p}(^m \ell_1,\dots,\ell_1,E;F)$ for any $m \geq 2$.
\end{theorem}
\begin{proof}

 Let $T\in\mathcal{L}\left(  ^{m}\ell_{1},\ldots,\ell_1,E;F\right).$  Let
 	$(x_{j}(k))_j\in \ell^w_{p}(\ell_{1})$ for all $k=1,\ldots,m-1$ and  $(y_{j})_j\in
 	\ell^w_{p}(E).$ By \cite{BoPe10} we know that
 	\[
 	\left(  x_{j_{1}}{(1)}\otimes\cdots\otimes x_{j_{m-1}}%
 	{(m-1)}\right)  _{j_{1},\ldots,j_{m-1}=1}^{\infty}\in \ell^w_{p}(\otimes_{\pi}%
 	^{m-1}\ell_{1})\subseteq \ell^w_{s}(\otimes_{\pi}^{m-1}\ell_{1}).
 	\]
 	Let $\widetilde{T}:\left(  \widehat\otimes_{\pi}^{m-1}\ell_{1}\right)  \times
 	E\rightarrow F$ be the bilinear map associated to $T.$ Since $\widehat\otimes_{\pi
 	}^{m-1}\ell_{1}=\ell_{1}$, by hypothesis $\widetilde{T}$ is multiple $\left(
 	r;s\right)  $-summing. Thus%
 	\[
 	\left(  T\left(  x_{j_{1}}{(1)},\ldots,x_{j_{m-1}}{(m-1)},y_{j_{m}}\right)
 	\right)  _{j_{1},\ldots, j_{m}=1}^{\infty}\]
\[=\left(  \widetilde{T}\left(  x_{j_{1}%
 	}{(1)}\otimes\cdots\otimes x_{j_{m-1}}{(m-1)},y_{j_{m}}\right)
 	\right)  _{j_{1},\ldots ,j_{m}=1}^{\infty}\subseteq \ell_{r}(F).
 	\]

\end{proof}

We now show half of the theorem announced in the introduction.

\begin{theorem}\label{thm:multiplel1}
 Let $m\geq 1$, $Z$ a cotype 2 space. Then 
 $$\mathcal L(^{m+1}\ell_1,\dots,\ell_1,Z;\mathbb K)=\Pi^{mult}_{1,1}(^{m+1} \ell_1,\dots,\ell_1,Z;\mathbb K).$$ 
\end{theorem}
\begin{proof}
Thanks to Theorem \ref{thm:repeatedl1}, we may assume that $m=1$. Using \cite[Proposition 2.2]{DJT}, it suffices to show that for any continuous linear operators $u_{1}%
:c_{0}\rightarrow \ell_{1}$ and $u_{2}:c_{0}\rightarrow Z$ the map
$S:=T(u_{1},u_{2})$ is multiple $(1,1)$-summing.

Since $Z$ has cotype $2$, by \cite[Proof of Theorem 11.14]{DJT} we have
$u_{2}\in\Pi_{2,2}(c_{0},Z)$ and by Grothendieck's theorem, $u_{1}^{\ast}\in\Pi_{2,2}(\ell_{\infty},\ell_{1})$.
Let $S_2:c_{0}\rightarrow\mathcal{L}(c_{0},\mathbb{K})=\ell_{1}$ and $T_2:\ell_2\rightarrow\mathcal{L}(\ell_1,\mathbb{K})=\ell_{\infty}$ be the linear operators associated to $S$ and $T$ by the formulas
 $S_{2}(y)(x):=S(x,y)$ and $T_2(y)(x):=T(x,y)$. Since $S_{2}=u_{1}^{\ast}\circ T_{2}\circ u_{2}$, by the composition theorem $S_{2}\in\Pi
_{1,1}(c_{0},\ell_{1})$. Finally, \cite[Proposition 2.2]{PGVil03} allows us to conclude that $S$, hence $T$, are $(1,1)$-multiple summing.
\end{proof}

In particular, the previous theorem shows that any bilinear form on $\ell_1\times\ell_p$ with $1\leq p\leq 2$ is $(1,1)$-multiple summing.
On the other hand, Littlewood's theorem says that all bilinear forms on any $X\times Y$ is $(4/3,1)$-summing and a standard application 
of the Kahane-Salem-Zygmund inequality (see, for instance, \cite{ABPS}) shows that this cannot be improved if $X=\ell_p$ and $Y=\ell_q$ for $p,q\geq 2$. 
Thus it is natural to study the best (=smallest) $r$ such that any bilinear form on $\ell_p\times \ell_q$ is $(r,1)$-summing, or more
generally the best $r$ such that any $m$-linear form on $\ell_{p_1}\times\cdots\times \ell_{p_m}$ is $(r,1)$-summing. 

Unfortunately, it does not seem that interpolation works. For instance, interpolating between $\ell_1\times\ell_2$ and $\ell_1\times\ell_\infty$
seems difficult because $[\ell_1^{w,0}(\ell_2),\ell_1^{w,0}(\ell_\infty)]_\theta=[\ell_1\widehat\otimes_\veps \ell_2,\ell_1\widehat\otimes_\veps\ell_\infty]_\theta\neq l_1^{w,0}(\ell_p)$
for the appropriate $p$. 


At least, Theorem \ref{thm:optimultiple} below will show that, for all $p\geq 2$, the smallest $r\geq 1$ such that any bilinear form on $\ell_1\times\ell_p$ is multiple $(r,1)$-summing
satisfies $r\geq \frac{4p}{3p+2}$. 

Using the notion of coordinatewise summability, we can also solve the problem for $(m+1)$-linear forms on $\ell_1\times\ell_2\times\cdots\times\ell_2$.

\begin{proposition}
 Let $m\geq 1$. Then every $(m+1)$-linear form on $\ell_1\times\ell_2\times\cdots\times\ell_2$ is $\left(\frac{2m}{m+1},1\right)$-multiple summing, and this value
 is optimal.
\end{proposition}

\begin{proof}
We prove this result by induction. It is true for $m=1$ and assume that it is true until $m-1$. Let 
$T\in\mathcal L(^{m+1} \ell_1,\ell_2,\dots,\ell_2;\mathbb K)$. Let $k\in \{1,\dots,m+1\}$ and $x_k\in \ell_1$ if $k=1$, $x_k\in\ell_2$ otherwise.
Define $T_k(x_k)(y_1,\dots,y_m)\mapsto (y_1,\dots,y_{k-1},x_k,y_k,\dots,y_m)$ for $k=1,\dots,m+1$.
Each $T_k(x_k)$ is a $m-$linear form on $\ell_2\times\cdots\times\ell_2$ if $k=1$, on $\ell_1\times\ell_2\times\dots\times\ell_2$
otherwise. Let $r_1=\frac{2m}{m+1}$ and $r_k=\frac{2(m-1)}{m}$ for $k=2,\dots,m+1$. Then, by the Bohnenblust--Hille inequality \cite{BH31} (for $k=1$) and by the induction hypothesis (otherwise),
we know that each $T_k(x_k)$ is multiple $(r_k,1)-$summing. It follows from the results of \cite{PopSinna13} (see also \cite{DPS10} and \cite{BCW}) that
$T$ is multiple $(r,1)-$summing with 
$$r=\frac{2R}{m+R}\textrm{ where }R=\sum_{k=1}^{m+1}\frac{r_k}{2-r_k}=m^2.$$
The optimality  follows from the fact that the optimal $r$ such that any $m$-linear form on $\ell_2\times\cdots\times\ell_2$ is multiple $(r,1)-$summing
is $\frac{2m}{m+1}$. This cannot be improved if we add one factor, since we can take $S(x,z)=x_1T(z)$ where $x\in\ell_1$, $z\in\ell_2\times\cdots\times\ell_2$ and $T\in\mathcal L(^m\ell_2;\KK)$.
\end{proof}

\begin{remark}
The full bilinear version of Grothendieck's theorem says that any bilinear map from $\ell_1\times\ell_1$ into $\ell_2$ is multiple $(1,1)$-summing. We cannot improve
our result up to this point. Indeed, it is not possible that all bilinear maps from $\ell_1\times\ell_2$ into $\ell_2$ is multiple $(r,1)$-summing for some $r<2$.
This would imply that $Id_{\ell_2}$ is $(r,1)$-summing, which is not the case (see \cite[Theorem 10.5]{DJT}).
\end{remark}

\section{Cotype in summability theory of multilinear maps} \label{sec:cotype}

Since Grothendieck's theorem, coincidence results are very important in summability theory. They show conditions on the Banach 
spaces $X_1,\dots,X_m,Y$ and on $r,s$ in order to have $\mathcal L(^mX_1,\dots,X_m;Y)=\Pi_{r,s}^\Lambda(X_1,\dots,X_m;Y)$. 
It is known that having information on the cotype of the spaces may imply restrictions on the possible indices $r,s$ for such an equality to hold (see for instance 
\cite{bp0} or \cite{bpr0}). Our aim in this section is to provide further results in this direction. They will rely on the following deep result of Maurey and Pisier \cite{MP}:
let $Z$ be an infinite-dimensional Banach space and let $q_Z:=\inf\{q;\ Z$ has cotype $q\}$. Then $\ell_{q_Z}$ is finitely representable in $Z$:
for all $n\geq 1$, there exists $Z_n\subset Z$ and an isomorphism $S_n:\ell_{q_Z}^n\to Z_n$ such that $\|S_n\|\cdot \|S_n^{-1}\|\leq 2$.

Without loss of generality, we may and shall assume that $\|S_n\|\leq 1$. We will set, for $1\leq i\leq n$, $y_i=S_n(e_i)$ and $y_i^*=(S_n^*)^{-1}(e_i)$. We shall also extend $y_i^*$ to the whole $Z$
by the Hahn-Banach theorem. We shall use repeatedly that $\|y_i^*\|\leq 2$, that $1/2\leq\|y_i\|\leq 1$ and the following lemma.

\begin{lemma}
 Let $s\geq 1$ be such that $s^*\geq q_Z$. Then $\omega_{s}\left((y_i)_{i=1}^n\right)\leq n^{\frac 1{q_Z}-\frac1{s^*}}$.
\end{lemma}
\begin{proof}
 We just observe that 
 \begin{align*}
\omega_{s}\left( (y_i)_{i=1}^n\right)   & =\sup_{z^*\in B_{Z^*}}\left(\sum_{i=1}^n |\langle z^*,y_i\rangle|^s\right)^{1/s}\\
&\leq \sup_{z^*\in B_{Z^*}}\sup_{a\in B_{\ell_{s^*}^n}}\big|\sum_{i=1}^n a_i \langle z^*,y_i\rangle\big|\\
&\leq \sup_{a\in B_{\ell_{s^*}^n}}\left\|\sum_{i=1}^n a_i y_i\right\|\\
&\leq n^{\frac 1{q_Z}-\frac1{s^*}}.
\end{align*}
\end{proof}

\subsection{Cotype and multiple summability}

	We shall prove Theorem \ref{thm:optimultiple} which assures that Theorem
\ref{thm:multiplel1} is sharp. We need a lemma that in particular gives
an alternative probabilistic proof (perhaps already known) of the famous result
proved in 1947 by MacPhail \cite{mc} showing that in $\ell_{1}$ there is an
unconditionally summable sequence that fails to be absolutely summable (the
ideas of MacPhail were crucial for the development of the Dvoretzky--Rogers
Theorem in 1950):

\begin{lemma}
	\bigskip\label{macp} Let $n\geq 1$.There exists a sequence $(x_i)_{i=1}^n$ of elements in $\ell_{1}^n$, $x_i=\left(  \varepsilon
	_{i,j}\right)  _{j=1,\ldots,n}$,  with $\left\vert \varepsilon
	_{i,j}\right\vert =1$ and $\omega_{1}\left(  \left(  x_i\right)_{i=1}^n  \right)  \leq
	Cn^{3/2}$ for some constant $C.$
\end{lemma}

\begin{proof}
	Define $x_i=\left(  \varepsilon_{i,j}\right)  _{j=1,\ldots,n}$, $\left\vert
	\varepsilon_{i,j}\right\vert =1$. We show that we may choose the signs
	$\varepsilon_{i,j}$ so that $\omega_{1}\left(  \left(  x_i\right)_{i=1}^n  \right)  \leq
	Cn^{3/2}.$ In fact,%
	\begin{align*}
	\omega_{1}\left(  \left(  x_i\right)_{i=1}^n  \right)   &  =\sup_{\left\Vert
		\zeta\right\Vert _{\infty}\leq1}\sum\limits_{i=1}^{n}\left\vert \sum
	\limits_{j=1}^{n}\zeta_{j}\varepsilon_{i,j}\right\vert \\
	&  =\sup_{\left\Vert \zeta\right\Vert _{\infty}\leq1}\sup_{\left\Vert
		w\right\Vert _{\infty}\leq1}\sum\limits_{i,j=1}^{n}\zeta_{j}w_{i}%
	\varepsilon_{i,j}.
	\end{align*}
	By the Kahane-Salem-Zygmund inequality we can find $\varepsilon_{i,j}%
	\in\{-1,1\}$ such that%
	\[
	\sup_{\left\Vert \zeta\right\Vert _{\infty}\leq1}\sup_{\left\Vert w\right\Vert
		_{\infty}\leq1}\sum\limits_{i,j=1}^{n}\zeta_{j}w_{i}\varepsilon_{i,j}\leq
	Cn^{3/2}.
	\]
	
\end{proof}

As we mentioned, the above lemma provides an alternative proof of MacPhail's theorem, since it
is immediate from the lemma that $Id_{\ell_{1}}$ cannot be absolutely $\left(
r;1\right)  $-summing for $r<2.$

\begin{theorem}
	\label{thm:optimultiple} Let $m\geq1$, $Z$ be an infinite-dimensional Banach
	space and $r\geq1$. Assume that $\mathcal{L}(^{m+1}\ell_{1},\dots,\ell
	_{1},Z;\mathbb{K})=\Pi_{r,1}^{mult}(^{m+1}\ell_{1},\dots,\ell_{1}%
	,Z;\mathbb{K}).$ Then $r\geq\frac{4q_{Z}}{3q_{Z}+2}$.
\end{theorem}

\begin{proof}
	We need only to consider the case $m=1$ and we suppose that every bounded
	bilinear form $T:\ell_{1}\times Z\rightarrow\mathbb{K}$ is $(r,1)$-multiple
	summing. Let $n\geq1$ and let $Z_{n},S_{n},(y_{i})_{1\leq i\leq n}$ and
	$(y_{i}^{\ast})_{1\leq i\leq n}$ as above. Let us now consider
	\begin{align*}
	T:\ell_{1}\times Z &  \rightarrow\mathbb{K}\\
	\left(  x,y\right)   &  \mapsto\sum_{i=1}^{n}x_{i}\langle y_{i}^{\ast
	},y\rangle
	\end{align*}
	which satisfies $\Vert T\Vert\leq2$. We choose a sequence $(x_i)_i\in\ell_{1}$
	as in Lemma \ref{macp}. 
 It follows
	that
	\begin{align*}
	n^{2/r} &  =\left(  \sum\limits_{i,j=1}^{n}\left\vert T\left(
	x_i,y_j\right)  \right\vert ^{r}\right)  ^{1/r}\leq C\omega_{1}\left(
	(x_i)_{i=1}^n\right)  \omega_{1}\left(  (y_j)_{j=1}^n\right)  \\
	&  \leq Cn^{3/2}n^{1/{q_{Z}}}%
	\end{align*}
	and since $n$ is arbitrary%
	\[
	r\geq\frac{4q_{Z}}{3q_{Z}+2}.
	\]
	
\end{proof}

\begin{remark}
Consider $1<p\leq 2$, $Z=\ell_{p^*}$ and take $S\in \mathcal L(^{m+1}\ell_1,\ldots,\ell_1,\ell_{p^*};\mathbb K)$. The $m$-linear operator $S_m:\ell_1\times \cdots\times \ell_1\to\ell_p$ given by $S_m(x_1,\ldots,x_m)(x)=S(x_1,\ldots,x_m,x)$, $x_1,\ldots,x_m\in\ell_1$, $x\in\ell_{p^*}$, is continuous. Hence, by  \cite[Corollary 4.3]{BoPe10} 
$$S_m\in \mathcal L(^{m}\ell_1,\ldots,\ell_1;\ell_p)=\Pi_{s,1}^{mult} (^m\ell_1,\ldots,\ell_1;\ell_p)$$ for any $s\geq \frac{2p}{3p-2}$. From \cite[Proposition 2.2]{PGVil03}, it follows that $S\in \Pi_{s,1}^{mult} (^{m+1}\ell_1,\ldots,\ell_1,\ell_{p^*};\mathbb K)$. In that case, 
$$
\frac{4q_{\ell_{p^*}}}{3q_{\ell_{p^*}}+2}=\frac{4p^*}{3p^*+2}=\frac{4p}{5p-2}<\frac{2p}{3p-2}.
$$
It is of interest to know if there is $\frac{4q_{\ell_{p^*}}}{3q_{\ell_{p^*}}+2}\leq r<\frac{2p}{3p-2}$ and a $m+1$-linear form $T\in \mathcal L(^{m+1}\ell_1,\ldots,\ell_1,\ell_{p^*};\mathbb K)$ such that $T$ is not $(r,1)$-multiple summing.
\end{remark}

In Theorem \ref{thm:multiplel1} we have proved that if $Z$ is a cotype 2 space then $\mathcal L(^{m+1}\ell_1,\ldots,\ell_1,Z;\mathbb K)=\Pi_{1,1}^{mult} (^{m+1}\ell_1,\ldots,\ell_1, Z;\mathbb K)$. Theorem \ref{thm:optimultiple} gives a partial converse:

\begin{corollary}
Let $Z$ be a Banach space and $m\in \mathbb N$. If $\mathcal L(^{m+1}\ell_1,\ldots,\ell_1,Z;\mathbb K)=\Pi_{1,1}^{mult} (^{m+1}\ell_1,\ldots,\ell_1, Z;\mathbb K)$ then $q_Z= 2$.
\end{corollary}

\begin{proof}
It follows from Theorem \ref{thm:optimultiple} as $1\geq \frac{4q_{Z}}{3q_{Z}+2}$ if and only if $q_Z\leq 2$.
\end{proof}

\subsection{Cotype and absolute summability}

Several results already appeared in the literature linking cotype and absolute summability. We shall use the following result of Botelho \cite{bot} which gives sufficient conditions for a multilinear map to be absolutely summing under cotype conditions.

\begin{lemma}\label{lem:botelho}
Let $Z_1,\dots,Z_m,W$ be Banach spaces. 
\begin{enumerate}
\item Assume that each $Z_i$ has cotype $q_i$. Then $\mathcal L(^mZ_1,\dots,Z_m;W)=\Pi_{r,1}^{abs}(^m Z_1,\dots,Z_m;W)$ as soon as $\frac 1r\leq \sum_{j=1}^m \frac 1{q_j}$.
\item Assume that $W$ has cotype $q$. Then $\mathcal L(^mZ_1,\dots,Z_m;W)=\Pi_{q,1}^{abs}(^ m Z_1,\dots,Z_m;W)$.
\end{enumerate}
\end{lemma}

Conversely, we shall prove the following result giving a sufficient condition under cotype condition.
\begin{theorem}\label{thm:absolute}
Let $l\geq 0$, $m\geq 1$, $p_1,\dots,p_l\geq 1$, $Z_{l+1},\dots,Z_m$ be infinite-dimensional Banach spaces and let $r>0$ be such that 
$$\mathcal L(^m \ell_{p_1},\dots,\ell_{p_l},Z_{l+1},\dots,Z_m{;W})=\Pi_{r,1}^{abs}(^m \ell_{p_1},\dots,\ell_{p_l},Z_{l+1},\dots,Z_m{;W}).$$
\begin{enumerate}
\item Assume that $W$ is infinite-dimensional. Then 
$$\frac 1r\leq \max\left(\frac 1{q_W}-\sum_{j=1}^l \frac1{p_j},0\right)+\sum_{j=1}^l \frac 1{p_j}+\sum_{j=l+1}^m \frac{1}{q_{Z_j}}.$$
\item Assume that $W$ is finite-dimensional. Then 
$$\frac 1r\leq \max\left(1-\sum_{j=1}^l \frac1{p_j},0\right)+\sum_{j=1}^l \frac 1{p_j}+\sum_{j=l+1}^m \frac{1}{q_{Z_j}}.$$
\end{enumerate}
\end{theorem}
\begin{proof}
Let $n\geq 1$. We first assume that $W$ is infinite-dimensional. 
We know that $\ell_{q_{W}}$ is finitely representable in $W$ and
	we denote by $(w_{i})_{i=1}^{n}$ the associated sequence in $W$. Let
	$y(j)=\left(  y_{i}(j)\right)  _{i=1}^{n}$ and $y^{\ast}(j)=\left(
	y_{i}^{\ast}(j)\right)  _{i=1}^{n}$ be, as usual in this paper, the vectors in
	$Z_{j}$ given by the Maurey--Pisier Theorem, $j=l+1,\ldots,m$. Consider the
	$m$-linear operator $T$ defined on $\ell_{p_1}\times \cdots\times\ell_{p_l}\times Z_{l+1}\times\cdots\times Z_{m}$ by 
	$$T
	\left(  z{(1)},\ldots,z{(m)}\right)   =\sum_{i=1}^{n}{}z_{i}%
	{(1)}\cdots z_i(l)\langle y_{i}^{\ast}(l+1),z{(l+1)}\rangle\cdots\langle y_{i}^{\ast}%
	(m),z{(m)}\rangle w_{i}   .$$
	It satisfies
	\begin{align*}
	\Vert T\Vert &  =\sup_{\|z(i)\|\leq 1}\left\Vert \sum_{i=1}^{n}z_{i}%
	{(1)}\cdots z_i(l)\langle y_{i}^{\ast}(l+1),z{(l+1)}\rangle\cdots\langle y_{i}^{\ast}%
	(m),z{(m)}\rangle w_{i}\right\Vert
	\\
	&  \leq\sup_{\|z(i)\|\leq 1}\left(  \sum_{i=1}^{n}{}\left\vert z_{i}{(1)}\cdots z_i(l)\langle y_{i}^{\ast
	}(l+1),z{(l+1)}\rangle\cdots\langle y_{i}^{\ast}(m),z{(m)}\rangle\right\vert
	^{q_{W}}\right)  ^{1/q_{W}}.
	\end{align*}
	We then apply Holder's inequality to get that 
	$$\|T\|\leq n^{ \max\left(\frac 1{q_W}-\sum_{j=1}^l \frac1{p_j},0\right)}2^{m-l}.$$
	On the other hand, we also have 
		\begin{align*}
	\frac{1}{2}n^{1/r}  &  \leq\left(  \sum\limits_{i=1}^{n}\left\Vert
	w_{i}\right\Vert ^{r}\right)  ^{1/r}=\left(  \sum\limits_{i=1}^{n}\left\Vert
	T\left(  e_i,\dots,e_i,y_{i}(l+1),\ldots,y_{i}(m)\right)  \right\Vert ^{r}\right)  ^{1/r}\\
	&\leq
	C\left\Vert T\right\Vert \omega_1(e_i)\cdots\omega_1(e_i) \omega_{1}\left(  y(l+1)\right)  \cdots \omega_{1}\left(
	y(m)\right) \\
	&  \leq C\|T\|n^{\sum_{j=1}^l \frac 1{p_j}+\sum_{j=l+1}^m \frac{1}{q_{Z_j}}}.
	\end{align*}
	Since $n$ is arbitrary, we get the desired inequality.
	
	If $W$ is finite-dimensional, we may assume that $W=\mathbb K$ and we now set 
$$T
	\left(  z{(1)},\ldots,z{(m)}\right)   =\sum_{i=1}^{n}{}z_{i}%
	{(1)}\cdots z_i(l)\langle y_{i}^{\ast}(l+1),z{(l+1)}\rangle\cdots\langle y_{i}^{\ast}%
	(m),z{(m)}{\rangle.}$$
	The proof is completely similar, except that now we have 
		$$\|T\|\leq n^{ \max\left( 1-\sum_{j=1}^l \frac1{p_j},0\right)}2^{m-l}.$$
\end{proof}
Combining the two previous results, we get several corollaries:
\begin{corollary}
	Let $m\geq1$ and $W$ be an infinite dimensional Banach space. Then for any
	infinite dimensional Banach spaces $Z_{j}$ with cotype $q_{Z_{j}},$ for all
	$j=2,...,m,$ we have
	\[
	\mathcal{L}(^{m}\ell_{2},Z_{2},\ldots,Z_{m};W)=\Pi_{r,1}^{abs}(^{m}\ell
	_{2},Z_{2},\ldots,Z_{m};W)\Leftrightarrow\frac{1}{r}\leq\frac{1}{2}+%
	{\displaystyle\sum\limits_{j=2}^{m}}
	\frac{1}{q_{Z_{j}}}.
	\]
\end{corollary}

\begin{corollary}
Let $m\geq p\geq 2$. Then $\mathcal L(^m \ell_p;\mathbb K)=\Pi_{r,1}^{abs}(^m\ell_p;\mathbb K)$ if and only if $r\geq \frac{p}m$. 
\end{corollary}

An important result in the theory of absolutely summing multilinear operators
is the Defant-Voigt result (\cite{AM}) which says that $\mathcal{L}(^{m}%
X_{1},\dots,X_{m};\mathbb{K})=\Pi_{r,1}^{abs}(X_{1},\dots,X_{m};\mathbb{K})$
for all Banach spaces $X_{1},\dots,X_{m}$ and all $r\geq1$. In \cite{araujo},
it was observed that this is optimal in the following sense: $\mathcal{L}(^{m}
c_{0},\dots,c_{0};\mathbb{K})=\Pi_{r,1}^{abs}(c_{0},\dots,c_{0};\mathbb{K})$
if and only if $r\geq1$. We can extend this to spaces with no finite cotype.

\begin{corollary}
	\label{jjj}Let $m\geq1$, let $Z$ be a Banach space with no finite cotype.
	Then
	\begin{equation*}
	\mathcal{L}(^{m}Z,...,Z;\mathbb{K})=\Pi_{r,1}^{abs}(^{m}Z,...,Z;\mathbb{K})
	\label{oooo}%
	\end{equation*}
	if and only if $r\geq1.$ Moreover, the result is sharp in the following
	sense: if $Z$ is any Banach space with finite cotype $q$, then for any $0<r<1$
	\[
	\mathcal{L}(^{m}Z,...,Z;\mathbb{K})=\Pi_{r,1}^{abs}(^{m}Z,...,Z;\mathbb{K})
	\]
	for all $m\geq q/r.$
\end{corollary}

\subsection{Cotype and absolute summability of polynomials }

Let $\mathcal P(^mX;Y)$ denote the space of all continuous $m$-homogeneous polynomials between the Banach spaces $X$ and $Y$, endowed with the usual sup norm.
We recall that $P\in \mathcal P(^mX;Y)$ is absolutely
$\left(  r,s\right)  $-summing (in symbols $P\in \mathcal P_{as(r,s)}(^mX;Y)$) if the sequence $\left(  P(x_{j}\right)  )_{j=1}^\infty$
belongs to $\ell_{r}(Y)$ whenever $(x_{j})_{j=1}^\infty$ is weakly $s$-summable in $X$ (see
\cite{AM}). In \cite[Theorem 3.1]{BPR10} it is proved that if $m$ is an even
integer, if $Z$ is an infinite dimensional real Banach space and if $r<1$,
then the coincidence  $\mathcal{P}\left(  ^{m}Z;\mathbb{R}\right)
=\mathcal{P}_{as(r,s)}\left(  ^{m}Z;\mathbb{R}\right)  $  implies that $Id_{Z}$
is $\left(  \frac{mr}{1-r},s\right)  $-summing. A careful examination of
\cite[Theorem 3.1]{BPR10} shows that the argument of \cite{BPR10} cannot be
extended to the case of odd integers and complex scalars. Our method allows us
to provide a proof working in all cases when the space assumes the infimum of its cotypes.

\begin{theorem}
	Let $m\geq1$, let $Z$ be an infinite dimensional Banach space and let
	$r\in(0,1)$. If
	\begin{equation}
	\mathcal{P}\left(  ^{m}Z;\mathbb{K}\right)  =\mathcal{P}_{as(r,s)}\left(
	^{m}Z;\mathbb{K}\right)  ,\label{est11}%
	\end{equation}
	then
	\[
	q_{Z}\leq\frac{mrs^{\ast}}{s^{\ast}(1-r)+mr}.
	\]
	
\end{theorem}

\begin{proof} Let $n\geq 1$. 
	We still use the same notations for the finite representation of $\ell_{q_{Z}%
	}$ into $Z$. If $s^{\ast}\leq q_{Z}$ it is simple to verify that%
	\begin{equation}
	\omega_{s}\left(  (y_i)_{i=1}^n\right)  \leq1.\label{est0}%
	\end{equation}
	We set now
	\begin{align*}
	P:Z &  \rightarrow\mathbb{K}\\
	z &  \mapsto\sum_{i=1}^{n}{}\langle y_{i}^{\ast},z\rangle^{m}%
	\end{align*}
	which satisfies
	\[
	\Vert P\Vert\leq2^{m}n.
	\]

	From (\ref{est0}) we conclude that (\ref{est11}) implies $s^{\ast}>q_{Z}.$ In
	fact, if $s^{\ast}\leq q_{Z}$ we have
	\begin{align*}
	n^{1/r}  & \leq\left(  \sum\limits_{i=1}^{n}\left\vert P\left(  y_i\right)
	\right\vert ^{r}\right)  ^{1/r}\leq C\left\Vert P\right\Vert 
	\omega_{s}\left( ( y_i)_{i=1}^n\right) ^{m}\\
	& \leq C  2^{m}n
	\end{align*}
	and this would imply $r\geq1$, a contradiction. Therefore $s^{\ast}>q_{Z}$
	and
	\begin{align*}
	n^{1/r} &  \leq\left(  \sum\limits_{i=1}^{n}\left\vert P\left(  y(i)\right)
	\right\vert ^{r}\right)  ^{1/r}\leq C\left\Vert P\right\Vert\omega_{s}\left( ( y_i)_{i=1}^n\right)  ^{m}\\
	&  \leq C\left(  2^{m}n\right)  n^{m\left(  \frac{1}{q_{Z}}-\frac{1}{s^{\ast}%
		}\right)  }.
	\end{align*}
	Since $n$ is arbitrary,
	\[
	q_{Z}\leq\frac{mrs^{\ast}}{s^{\ast}(1-r)+mr}.
	\]
	
\end{proof}

\begin{corollary}
	Let $m\geq1$, let $Z$ be an infinite dimensional Banach space that has cotype $q_z$ and let
	$r\in(0,1)$. If
	\[
	\mathcal{P}\left(  ^{m}Z;\mathbb{K}\right)  =\mathcal{P}_{as(r,s)}\left(
	^{m}Z;\mathbb{K}\right),
	\]
then $Id_{Z}$ is $\left(  \frac{mr}{1-r},s\right)
	$-summing.
\end{corollary}

\begin{proof}
By the previous theorem we know that $Id_{Z}$ is $\left(  \frac{mrs^{\ast}%
}{s^{\ast}\left(  1-r\right)  +mr},1\right)  $-summing, and by the Inclusion
Theorem for summing operators (see \cite[Theorem 10.4]{DJT}) we conclude
that it is $\left(  \frac{mr}{1-s},s\right)  $-summing.

\end{proof}

\providecommand{\bysame}{\leavevmode\hbox to3em{\hrulefill}\thinspace}
\providecommand{\MR}{\relax\ifhmode\unskip\space\fi MR }
\providecommand{\MRhref}[2]{%
  \href{http://www.ams.org/mathscinet-getitem?mr=#1}{#2}
}
\providecommand{\href}[2]{#2}

\bibliographystyle{amsplain} 
\providecommand{\bysame}{\leavevmode\hbox to3em{\hrulefill}\thinspace}
\providecommand{\MR}{\relax\ifhmode\unskip\space\fi MR }
\providecommand{\MRhref}[2]{%
	\href{http://www.ams.org/mathscinet-getitem?mr=#1}{#2}
}
\providecommand{\href}[2]{#2}

\end{document}